\documentclass[11pt]{article}
\usepackage{amsthm}
\usepackage{amsmath}
\usepackage{amssymb}
\usepackage{latexsym}
\usepackage{amsfonts}
\usepackage{placeins}
\usepackage{caption}
\usepackage{mathrsfs}
\usepackage[latin1]{inputenc}
\usepackage{listings}
\newtheorem{theorem}{\bf Theorem}[section]

\newtheorem{lemma}[theorem]{\bf Lemma}
\begin{document}
\title{Common terms of generalized Pell and Narayana's cows sequences}
\author{Bibhu Prasad Tripathy and Bijan Kumar Patel}
\date{}
\maketitle
\begin{abstract}
For an integer $k \geq 2$, let $\{ P_{n}^{(k)} \}_{n}$ be the $k$-generalized Pell sequence which starts with $0, \dots,0,1$($k$ terms) and each term afterwards is the sum of $k$ preceding terms. In this paper, we find all the solutions of the Diophantine equation $P_{n}^{(k)} = N_{m}$ in non-negative integers $(n, k, m)$ with $k \geq 2$, where $\{ N_{m} \}_m$ is the Narayana's cows sequence. Our approach utilizes the lower bounds for linear forms in logarithms of algebraic numbers established by Matveev, along with key insights from the theory of continued fractions.
\end{abstract} 


\section{Introduction}
The Pell sequence $\{P_n\}_{n\geq 0}$ is a binary recurrence sequence  given by 
\[
P_{n+2} = 2P_{n+1} + P_{n}~~ {\rm for} ~n\geq 0,
\]
with initials $P_{0} = 0$ and $P_{1} = 1$. 

Let $k \geq 2$ be an integer. We consider a generalization of the Pell sequence known as the $k$-generalized Pell sequence, $\{P_{n}^{(k)} \}_{n \geq -(k-2)}$ is given by the recurrence
\begin{equation}\label{eq 1.1}
P_{n}^{(k)} = 2 P_{n-1}^{(k)} +P_{n-2}^{(k)} +\dots+P_{n-k}^{(k)}~\text{for all} \quad n\geq 2,  
\end{equation}
with initials $P_{-(k-2)} ^{(k)} = P_{-(k-3)} ^{(k)} =\dots= P_{0} ^{(k)} = 0 $ and $P_{1} ^{(k)} = 1$. We shall refer to $P_{n}^{(k)}$  as the $n$th $k$-Pell number.  We see that this generalization is a family of sequences, with each new choice of $k$ producing a unique sequence. For example, if $k = 2$, we get $P_{n}^{(2)} = P_{n}$, the $n$th Pell number.

The Narayana's cows sequence $\{N_{m}\}_{m \geq 0}$ is a ternary recurrent sequence given by 
\[
N_{m+3} = N_{m+2} + N_{m}~~ {\rm for} ~m\geq 0,
\]
with initials $N_{0} = N_{1} = N_{2} = 1$. It is the sequence A000930 in the OEIS. Its first few terms are 
\[
1, 1, 1, 2, 3, 4, 6, 9, 13, 19, 28, 41, \dots
\]

There are many literature in number theory on finding the intersection of two recurrent sequences of positive integers. Recently, researchers have taken a keen interest in the challenge of determining the intersection between a $k$-generalized Pell sequence and various other number sequences. For example, one can go through  \cite{Bravo1, Bravo3, Normenyo, Normenyo1}. The objective of this paper is to find all the Narayana numbers in $k$-generalized Pell sequence. To accomplish this, we solve the Diophantine equation
\begin{equation}\label{eq 1.2}
    P_{n}^{(k)} = N_{m}.
\end{equation}
In particular, our main result is the following.
\begin{theorem}\label{thm1}
 All the solutions of the Diophantine equation \eqref{eq 1.2} in positive integers with $k \geq 2$ are given by 
\[
 P_{1}^{(k)} = N_{0} = N_{1} = N_{2}, \quad P_{2}^{(k)} = N_{3}, \quad \text{and} \quad P_{6}^{(4)} = N_{13}
\]
except in the cases $k \geq 3$ which we can additionally have $P_{4}^{(k)} = N_{8}$.
\end{theorem}

To establish the proof of Theorem \ref{thm1}, we first find an upper bound for $n$ in terms of $k$ by applying Matveev's result on linear forms in logarithms \cite{Matveev}. When $k$ is small, the theory of continued fractions suffices to lower such bounds and complete the calculations. When $k$ is large, we use the fact that the dominant root of the $k$-generalized Pell sequence is exponentially close to $\phi^2$ {\rm{(see \cite{Bravo}, Lemma 2)}} where $\phi$ denotes the golden section. So we use this estimation in our further calculation with linear forms in logarithms to obtain absolute upper bounds for $n$ which can be reduced by using  Dujella and Peth\"{o}'s result \cite{Dujella}.  In this way, we complete the proof of our main result. Our proof relies on a few preliminary results, which are extensively discussed in the subsequent section.

\section{Preliminary Results}
\subsection{Properties of $k$-generalized Pell sequence}
The characteristic polynomial of the $k$-generalized Pell sequence is  
 \[
 \Phi_{k}(x) = x^k - 2 x^{k-1}  - x^{k-2} - \dots - x -1. 
 \]
The above polynomial is irreducible over $\mathbb{Q} [x]$ and it has one positive real root that is $\gamma := \gamma(k)$ which is located between $ \phi^ 2(1 - \phi ^{-k})$ and $\phi^2$, lies outside the unit circle (see \cite{Bravo2}). The other roots are firmly contained within the unit circle. To simplify the notation, we will omit the dependence on $k$ of $\gamma$ whenever no confusion may arise.

The Binet  formula for $P_{n}^{(k)}$ found in \cite{Bravo2} is
\begin{equation}\label{eq 2.3}
	P_{n}^{(k)} = \displaystyle\sum_{i=1}^{k} g_{k} (\gamma_{i})\gamma_{i}^{n} ,
\end{equation}
where  $\gamma_{i}$ represents the roots of the characteristic polynomial $\Phi_{k}(x)$ and the function $g_{k}$ is  given by
\begin{equation}\label{eq 2.4}
	g_{k}(z) := \frac{z-1}{(k+1)z^2 - 3kz + k -1}, 
\end{equation}
for an integer $k$ $\geq$ 2. Additionally,
it is  also shown in \cite[Theorem~3.1]{Bravo2} that the roots located inside the unit circle have a very minimal influence on the formula \eqref{eq 2.3}, which is given by the approximation 
\begin{equation}\label{eq 2.5}
	\left| P_{n}^{(k)} -  g_{k} (\gamma) \gamma^{n} \right| < \frac{1}{2} \quad \text{holds  for   all} \quad n \geq 2 - k.
\end{equation}
Therefore, for $n \geq 1$ and $k \geq 2$, we have
\begin{equation}\label{eq 2.6}
   P_{n}^{(k)} =  g_{k} (\gamma) \gamma^{n} + e_{k}(n), \quad \text{where} \quad |e_{k}(n)| \leq \frac{1}{2}. 
\end{equation}
Furthermore, it is shown  in \cite[Theorem~3.1]{Bravo2} that the inequality
\begin{equation}\label{eq 2.7}
	\gamma ^ {n-2} \leq P_{n} ^ {(k)} \leq \gamma ^ {n-1} \text{ holds for all } n \geq 1.
\end{equation}
The following result was proved by Bravo and Hererra \cite{Bravo1}.

\begin{lemma}\label{lem 2.1}{\rm{(\cite{Bravo1}, Lemma 2.1)}}.
Let $k \geq 2 $  be an integer. Then we have
\[
0.276 < g_{k}(\gamma) < 0.5 \  and \  \left| g_{k}(\gamma_{i}) \right| < 1 \quad for \quad 2 \leq i \leq k.
\]	
\end{lemma}
\noindent
Furthermore, they showed that the logarithmic height of $g_{k} (\gamma)$ satisfies
\begin{equation}\label{eq 2.8}
    h(g_{k}(\gamma)) < 4k\log \phi + k \log(k+1) \quad \text{for all} \quad k \geq 2.
\end{equation}

\begin{lemma}\label{lem 2.2}{\rm{(\cite{Bravo}, Lemma 2)}}.
If $k \geq 30$ and $n \ge 1$ are integers that satisfies $n < \phi^{k/2}$, then
\begin{equation}\label{eq 2.9}
   g_{k} (\gamma) \gamma^{n} = \frac{\phi^{2n}}{\phi + 2}(1 + \xi), \quad \text{where} \quad  |\xi| < \frac{4}{\phi^{k/2}}.
\end{equation}
\end{lemma}

\subsection{Properties of Narayana's cows sequence}
The characteristic polynomial of Narayana's cows sequence is 
\[
f(x) = x^{3} - x^{2} - 1,
\]
which is irreducible over $\mathbb{Q} [x]$ and has roots $\alpha$, $\beta$ and $\delta$ given by
\[
\alpha = \frac{1}{3} + \left( \frac{29}{54} + \sqrt{\frac{31}{108}} \right)^{1/3} + \left( \frac{29}{54} - \sqrt{\frac{31}{108}} \right)^{1/3}, 
\]
\[
\beta = \frac{1}{3} + w \left( \frac{29}{54} + \sqrt{\frac{31}{108}} \right)^{1/3} + w^{2} \left( \frac{29}{54} - \sqrt{\frac{31}{108}} \right)^{1/3}, 
\]
\[
\delta = \Bar{\beta} = \frac{1}{3} + w \left( \frac{29}{54} + \sqrt{\frac{31}{108}} \right)^{1/3} + w^{2} \left( \frac{29}{54} + \sqrt{\frac{31}{108}} \right)^{1/3}, 
\]
where $w = \frac{-1+i\sqrt{3}}{2}$. The Binet formula for the Narayana's cows sequence is given by
\begin{equation}\label{eq 2.10}
    N_{m} = p \alpha^{m} + q \beta^{m} + r \delta^{m} \quad \text{for all} \quad m \geq 0,
\end{equation}
where 
\[
p = \frac{\alpha}{(\alpha - \beta)(\alpha - \delta)}, \quad q = \frac{\beta}{(\beta - \alpha)(\beta - \delta)}, \quad r = \frac{\delta}{(\delta - \alpha )(\delta - \beta)}. 
\]
The formula \eqref{eq 2.10} can also be written in the form
\begin{equation}\label{eq 2.11}
N_{m} = C_{\alpha} \alpha^{m+2} + C_{\beta} \beta^{m+2} + C_{\delta} \delta^{m+2} \quad \text{for all} \quad m \geq 0,    
\end{equation}
where
\[
C_{x} = \frac{1}{x^{3} +2}, \quad x \in \{ \alpha, \beta, \delta \}.
\]
The coefficient $C_{\alpha}$ has the minimal polynomial $31x^{3} - 31x^{2} + 10x - 1$  over $\mathbb{Z}$ and
all the zeros of this polynomial lie strictly inside the unit circle.  One can approximate the following:
\[
\alpha \approx 1.46557; \  |\beta| = |\delta| \approx 0.826031; \  | C_{\beta} \beta^{m+2} + C_{\delta} \delta^{m+2}| < 1/2 \quad \text{for all} \quad m \geq 1.
\]
\begin{lemma}\label{lem 2.3}
For every positive integer $m \geq 1$, we have \begin{equation}\label{eq 2.12}
\alpha^{m-2} \leq N_{m} \leq \alpha^{m-1}. 
\end{equation}
\end{lemma}
\begin{proof}
This can be easily proved by the method of induction on $m$.
\end{proof}

\subsection{Linear forms in logarithms}
Let $\gamma$ be an algebraic number of degree $d$ with a minimal primitive polynomial 
\[
f(Y):= b_0 Y^d+b_1 Y^{d-1}+ \cdots +b_d = b_0 \prod_{j=1}^{d}(Y- \gamma^{(j)}) \in \mathbb{Z}[Y],
\]
where the $b_j$'s are relatively prime integers, $b_0 >0$, and the $\gamma^{(j)}$'s are conjugates of $\gamma$. Then the \emph{logarithmic height} of $\gamma$ is given by
\begin{equation}\label{eq 2.13}
h(\gamma)=\frac{1}{d}\left(\log b_0+\sum_{j=1}^{d}\log\left(\max\{|\gamma^{(j)}|,1\}\right)\right).
\end{equation}
With the above notation, Matveev (see  \cite{Matveev} or  \cite[Theorem~9.4]{Bugeaud}) proved the following result.

\begin{theorem}\label{thm2}
Let $\eta_1, \ldots, \eta_s$ be positive real algebraic numbers in a real algebraic number field $\mathbb{L}$ of degree $d_{\mathbb{L}}$. Let $a_1, \ldots, a_s$ be non-zero  integers such that
\[
\Lambda :=\eta_1^{a_1}\cdots\eta_s^{a_s}-1 \neq 0.
\]
Then
\[
- \log  |\Lambda| \leq 1.4\cdot 30^{s+3}\cdot s^{4.5}\cdot d_{\mathbb{L}}^2(1+\log d_{\mathbb{L}})(1+\log D)\cdot B_1 \cdots B_s,
\]
where
\[
D\geq \max\{|a_1|,\ldots,|a_s|\},
\]
and
\[
B_j\geq \max\{d_{\mathbb{L}}h(\eta_j),|\log \eta_j|, 0.16\}, ~ \text{for all} ~ j=1,\ldots,s.
\]
\end{theorem}

\subsection{Reduction method}
Here, we present the following result due to Dujella and Peth\"{o} \cite[Lemma~5 (a)]{Dujella} which is a generalization of a result of Baker and Davenport's result \cite{Baker}.
\begin{lemma}\label{lem 2.5}
Let $\widehat{\tau}$ be an irrational number, and let $A,C,\widehat{\mu}$ be some real numbers with $A>0$ and $C>1$. Assume that $M$ is a positive integer, and let $p/q$ be a convergent of the continued fraction of the irrational $\widehat{\tau}$ such that $q > 6M$. Put \[\epsilon:=||\widehat{\mu} q||-M||\widehat{\tau} q||,
\]
where $||\cdot||$ denotes the distance from the nearest integer.  If $\epsilon >0$, then there is no solution to the inequality 
\[
0< |r \widehat{\tau}-s+\widehat{\mu}| <AC^{-t},
\]
in positive integers $r$, $s$ and $t$ with
\[
r \leq M \quad\text{and}\quad t \geq \frac{\log(Aq/\epsilon)}{\log C}.
\]
\end{lemma}
\subsection{Useful Lemmas}
We conclude this section by recalling two lemmas that we will need in this work.

\begin{lemma}\label{lem 2.6} {\rm{(\cite{Weger}, Lemma 2.2)}}
Let $a, x \in \mathbb{R}$. If $0< a < 1$ and $|x| < a$, then 
\[
|\log(1+x)| < \frac{-\log(1-a)}{a}\cdot |x|
\]
and 
\[
|x| < \frac{a}{1 - e^{-a}} \cdot |e^{x} - 1|.
\]
\end{lemma}
\begin{lemma}\label{lem 2.7} {\rm{(\cite{Sanchez}, Lemma 7)}}
If $m \geq 1$, $S \geq (4m^{2})^{m}$ and $\frac{x}{(\log x)^{m}} < S$, then $x < 2^{m} S (\log S)^{m}$.
\end{lemma}

\section{Proof of Theorem \ref{thm1}}
Since  $P_{1}^{(k)} = 1 = N_{0} = N_{1} = N_{2}$, $P_{2}^{(k)} = 2 = N_{3}$ holds for $k \geq 2$ and $P_{4}^{(k)} = 13 = N_{8}$ holds for $k \geq 3$.  Therefore, we may assume that  $n \geq 5$. For $5 \leq  n \leq k+1$, we have that $P_{n}^{(k)} = F_{2n-1}$ where $F_{n}$ is the $n$th Fibonacci
number. So the equation \eqref{eq 1.2} becomes
\[
 F_{2n-1} = N_{m},
\]
which has no solution for $n \geq 5$ and $m \geq 0$. Therefore it has no solution in the range $5 \leq n \leq k+1$. From now, we assume that $n \geq k+2$ and $k \geq 2$. 
\subsection{An initial relation between $n$ and $m$ }
Combining the inequalities \eqref{eq 2.7} and \eqref{eq 2.12} together with equation \eqref{eq 1.2}, we have 
\[
\gamma^{n-2} \leq \alpha^{m-1} \quad \text{and} \quad \alpha^{m-2} \leq \gamma^{n-1}.
\]
Then, we deduce that  
\[
(n-2)\left( \frac{\log \gamma}{\log \alpha}  \right) +1  \leq m \leq (n-1) \left( \frac{\log \gamma}{\log \alpha} \right) + 2.
\]
Using the fact $\phi^ 2(1 - \phi ^{-k}) < \gamma(k) < \phi^{2}$ for all $k \geq 2$, it follows that
\begin{equation}\label{eq 3.14}
1.25n - 1.5 < m < 2.52n - 0.52.    
\end{equation}
\subsection{Upper bounds for $n$ and $m$ in terms of $k$}
By using \eqref{eq 1.2}, \eqref{eq 2.6} and \eqref{eq 2.11}, we obtain
\begin{equation*}
   g_{k}(\gamma) \gamma^{n} +e_{k}(n) = c_{\alpha} \alpha^{m+2} + c_{\beta} \beta^{m+2} + c_{\delta} {\delta}^{m+2}.
\end{equation*}
Taking absolute values on both sides of the above equality, it yields
\begin{equation}\label{eq 3.15}
\left| g_{k}(\gamma) \gamma^{n} - c_{\alpha} \alpha^{m+2} \right| < \frac{1}{2} + |c_{\beta} \beta^{m+2} + c_{\delta} {\delta}^{m+2}| < 1. 
\end{equation}
Dividing both sides of the above inequality by $c_{\alpha} \alpha^{m+2}$, we get
\begin{equation}\label{eq 3.16}
 \left| \left( c_{\alpha}^{-1} g_{k}(\gamma)\right) \gamma^{n}  \alpha^{-(m+2)} - 1 \right| <   \frac{2.4}{\alpha^{m}}.
\end{equation}
Let
\begin{equation}\label{eq 3.17}
\Lambda_{1} := \left( c_{\alpha}^{-1} g_{k}(\gamma)\right) \gamma^{n}  \alpha^{-(m+2)} - 1.
\end{equation}
From \eqref{eq 3.16}, we have 
\begin{equation}\label{eq 3.18}
    |\Lambda_{1}| < 2.4 \cdot \alpha^{-m}.
\end{equation}
Suppose that $\Lambda_{1}= 0$, then we get
\[
g_{k}(\gamma) = c_{\alpha} \alpha^{m+2}\gamma^{-n},
\]
which implies that $g_{k}(\gamma)$ is an algebraic integer, which is a contradiction. Hence $\Lambda_{1} \neq 0$. Therefore, we apply Theorem \ref{thm2} to get a lower bound for $\Lambda_{1}$ given by \eqref{eq 3.17} with the parameters:
\[
\eta_{1} := c_{\alpha}^{-1} g_{k}(\gamma), \quad \eta_{2} := \gamma , \quad \eta_{3} := \alpha,
\]
and
\[ a_{1}:= 1, \quad a_{2}:= n, \quad a_{3}:= -(m+2).
\]
Note that the algebraic numbers $\eta_{1}, \eta_{2}, \eta_{3}$ belongs to the field  $\mathbb{L} := \mathbb{Q}(\gamma, \alpha)$, so we can assume $d_{\mathbb{L}} = [\mathbb{L}:\mathbb{Q}] \leq 3k$. Since $h(\eta_{2}) = (\log \gamma) / k < 2 \log \phi/k$ and $h(\eta_{3}) = (\log \alpha) / 3$, it follows that 
\[
\max\{3kh(\eta_{2}),|\log \eta_{2}|,0.16\} = 6\log \phi := B_{2}
\]
and 
\[
\max\{3kh(\eta_{3}),|\log \eta_{3}|,0.16\} = k\log \alpha := B_{3}.
\]
Since $h(c_{\alpha}) = \frac{\log 31}{3}$. Therefore, by the estimate \eqref{eq 2.8} and the properties of logarithmic height, it follows that for all $k \geq 2$
\begin{align*}
 h(\eta_{1})  & \leq h(c_{\alpha}) + h(g_{k}(\gamma)) \\
 & <  \frac{\log 31}{3} + 4k\log \phi + k \log(k+1) \\
 & < 5.2 k \log k.  
\end{align*}
Thus, we obtain
\[
\max\{3kh(\eta_{1}),|\log \eta_{1}|,0.16\} = 15.6 k^2 \log k := B_{1}.
\]
In addition, by \eqref{eq 3.14} we can take $D:= 2.52 n + 2$. Then by Theorem \ref{thm2}, we have
\begin{equation}\label{eq 3.19}
\log |\Lambda_{1}| >   -1.432 \times 10^{11} (3k)^{2} (1+ \log3k) (1 +\log (2.52n +2) (15.6 k^{2} \log k) (6\log \phi) (k \log \alpha).   
\end{equation}
From the comparison of lower bound \eqref{eq 3.19} and upper bound \eqref{eq 3.18} of   $|\Lambda_{1}|$ gives us
\[
m \log \alpha - \log 2.4 < 2.22 \times 10^{13} k^{5} \log k (1 + \log 3k)(1+ \log (2.52n+ 2)).
\]
Using the facts
$1+\log 3k < 4.1 \log k$ for all $k \geq 2$ and $1+ \log (2.52n+ 2)) < 2.6 \log n$ for all $n \geq 4$, we conclude that 
\[
m < 6.23 \times 10^{14} k^{5} \log^{2} k \log n.
\]
Using \eqref{eq 3.14}, the last inequality becomes
\begin{equation}\label{eq 3.20}
    \frac{n}{\log n} < 5 \times 10^{14} k^{5} \log^{2} k.
\end{equation}
Thus, putting  $S := 5 \times 10^{14} k^{5} \log^{2} k$ in \eqref{eq 3.20} and using Lemma \ref{lem 2.7} together with $33.84 +5\log k + 2 \log \log k < 52.8 \log k$ for all $k \geq 2$, gives
\begin{align*}
    n &< 2 \left( 5 \times 10^{14} k^{5} \log^{2} k \right ) \log \left( 5 \times 10^{14} k^{5} \log^{2} k \right) \\
    &< (1 \times 10^{15} k^{5} \log^{2} k )(33.84 + 5\log k + 2 \log \log k ) \\
    &< 5.28 \times 10^{16} k^{5} \log^{3} k.
\end{align*}
The result established in this subsection is summarized in the following lemma.
\begin{lemma}\label{lem3.1}
If $(n, k, m)$ is an integer solution   of \eqref{eq 1.2} with $k \geq 2$ and $n \geq k+2$, then the inequalities 
\begin{equation}\label{eq 3.21}
0.39m < n < 5.28 \times 10^{16} k^{5} \log^{3} k
\end{equation}
hold.
\end{lemma}

\subsection{The case of small $k$}
In this subsection, we treat the cases when $k \in [2, 360]$. Here for each value of $k$, Lemma \ref{lem3.1} provides an absolute upper bound on $n$ which is very large and will be reduced by Lemma \ref{lem 2.5}. In order to apply Lemma \ref{lem 2.5}, let
\begin{equation}\label{eq 3.22}
\Gamma_{1} := n\log \gamma -(m+2) \log \alpha + \log \left( c_{\alpha}^{-1} g_{k}(\gamma)\right).   
\end{equation}
Then $e^{\Gamma_{1}} - 1 = \Lambda_{1}$, where $\Lambda_{1}$ is defined by \eqref{eq 3.17}. Therefore, \eqref{eq 3.18} implies that
\begin{equation}\label{eq 3.23}
    |e^{\Gamma_{1}} -1 | < \frac{2.4}{\alpha^{m}} < 0.77
\end{equation}
for $m \geq 3$. Choosing $a := 0.77$, we obtain the inequality
\[
|\Gamma_{1}| = |\log (\Lambda_{1} + 1)| < \frac{- \log(1 - 0.77)}{0.77} \cdot \frac{2.4}{\alpha^{m}} < \frac{4.59}{\alpha^{m}}
\]
by Lemma \ref{lem 2.6}. Thus, it follows that
\begin{equation*}
0 < \left|n\log \gamma -(m+2) \log \alpha + \log \left( c_{\alpha}^{-1} g_{k}(\gamma)\right) \right| <  \frac{4.59}{\alpha^{m}}. 
\end{equation*}
Dividing this inequality by $\log \alpha$, we get 
 \begin{equation}\label{eq 3.24}
 \left| n \left( \frac{\log \gamma}{\log \alpha}\right) - m + \left(  \frac{\log \left(c_{\alpha}^{-1} g_{k}(\gamma)\right)}{\log \alpha} - 2 \right) \right| < 12.1 \cdot \alpha^{-m}.   
\end{equation}
With 
\[ \widehat{\tau} = \widehat{\tau}(k) := \frac{\log \gamma}{\log \alpha} ,  \quad \widehat{\mu} = \widehat{\mu}(k):= \frac{\log \left(c_{\alpha}^{-1} g_{k}(\gamma)\right)}{\log \alpha} - 2, \quad A:= 12.1 \quad \text{and} \quad C:= \alpha, 
\]
equation \eqref{eq 3.24} becomes
\begin{equation}\label{eq 3.25}
    0 <  |n \widehat{\tau} - m + \widehat{\mu}| < A \cdot C^{-m}.
\end{equation}
Clearly $\widehat{\tau}$ is an irrational number. We take $ M_{k} := \lfloor 5.28 \times 10^{16} k^{5} \log^{3} k  \rfloor$ which is an upper bound on $n$. Then by  Lemma \ref{lem 2.5}  for each $ k \in [2, 360]$, we have that 
\[
m < \frac{\log(Aq/\epsilon)}{\log C},
\]
where $ q = q(k) > 6M_{k}$ is a denominator of a convergent of the continued fraction of $\widehat{\tau}$ with $ \epsilon = \epsilon(k):= \| \widehat{\mu} q\| - M_{k}\|\widehat{\tau} q\| > 0$. A computer search with \textit{Mathematica} found that for $k \in [2, 360]$, the maximum value of $\log(Aq/\epsilon)/ \log C $ is $\leq 329$. Therefore  $m \leq 329$ and $n \leq 265$, since $n < (m + 1.5)/1.25$.

Finally, a brute force search with \textit{Mathematica} to compare $P_{n}^{(k)}$ and $ N_{m}$ in the range
\[
2 \leq k \leq 360, \quad 4 \leq n \leq 265, \quad \text{and} \quad 3 \leq m \leq 329
\]
with $m < n/0.39$ provides the only solution $P_{6}^{(4)} = N_{13}$ for the equation \eqref{eq 1.2}. This concludes the analysis of the case $k \in [2, 360]$. 

\subsection{The case of large $k$}
We now suppose that $k > 360$ and note that for such $k$ we have
\[
0.39m < n < 5.28 \times 10^{16} k^{5} \log^{3} k < \phi^{k / 2}.
\]
Here, it follows from \eqref{eq 1.2}, \eqref{eq 2.9} and \eqref{eq 3.15} that 
\[ \left| \frac{\phi ^{2n}}{\phi + 2} - c_{\alpha} \alpha^{m+2} \right| < \left| g_{k}(\gamma) \gamma^{n} - c_{\alpha} \alpha^{m+2} \right| + \frac{\phi ^{2n}|\xi|}{\phi + 2} <  1 + \frac{4 \phi ^{2n}}{(\phi + 2)\phi^{k/2}}.
\]
Dividing both sides of the above inequality  by $\frac{\phi ^{2n}}{\phi + 2}$ and using the fact $1/\phi ^{2n} < 1/\phi^{k/2}$ for all $n \geq k+2$ yield 
\begin{equation}\label{eq 3.26}
\left|\left( c_{\alpha}(\phi + 2)\right) \phi ^{-2n} \alpha^{m+2} -1 \right| < \frac{(\phi + 2)}{\phi^{2n}} + \frac{4}{\phi^{k/2}} < \frac{7.62}{\phi^{k/2}}. 
\end{equation}
In order to use Theorem \ref{thm2}, we take
\[
(\eta_{1}, a_{1}) := (c_{\alpha}(\phi + 2), 1), \quad \quad (\eta_{2}, a_{2}) := (\phi, -2n),~ \quad {\rm and} ~  \quad (\eta_{3}, a_{3}) := (\alpha, m+2).
\]
The number field containing $\eta_{1}, \eta_{2}, \eta_{3}$ is  $\mathbb{L} := \mathbb{Q}(\phi, \alpha)$, which has degree $d_{\mathbb{L}} = [\mathbb{L}:\mathbb{Q}]=6$. Here 
\begin{equation}\label{eq 3.27}
    \Lambda_{2} := \left( c_{\alpha}(\phi + 2)\right) \phi ^{-2n} \alpha^{m+2} - 1,
\end{equation}
is nonzero.  In contrast to this, assume that $\Lambda_{2} = 0$, then we would get $\phi^{2n}/ \phi + 2 = c_{\alpha} \alpha^{m+2}$. Using the $\mathbb{Q}$-automorphism
$(\alpha, \beta)$ of the Galois extension $\mathbb{Q}(\phi, \alpha, \beta)$ over $\mathbb{Q}$  we get that $12 < \phi^{2n}/ \phi + 2 < |c_{\beta}| |\beta|^{m+2} < 1$, which is impossible. Hence $\Lambda_{2} \neq 0$. Moreover, since 
\[
h(\eta_{2}) = \frac{\log \phi}{2}, \quad \quad  h(\eta_{3}) = \frac{\log \alpha}{3}
\]
and
\[
h(\eta_{1}) \leq h(c_{\alpha}) + h(\phi) + 2 \log 2 < 2.8,
\]
it follows that $B_{1}:= 16.8, B_{2}:= 1.45$ and $B_{3}:= 0.77$. Also since $m <2.52 n$, we can take $D:= 2.52n +2$. Thus, taking into account inequality \eqref{eq 3.26} and applying Theorem \ref{thm2}, we obtain 
\[
\frac{k}{2} \log \phi - \log 7.62 <  2.7 \times 10^{14} \times (1 + \log(2.52 n +2)).
\]
This implies that 
\[
k < 1.58 \times 10^{15} \log n,
\]
where $1 + \log(2.52 n +2) < 1.4 \log n$ for $n \geq k+2 > 362$. By using Lemma \ref{lem3.1} and the fact $38.51 + 5 \log k +3 \log (\log k) < 12.5 \log k$ for $k > 360$, we get 
\begin{align*}
k & < 1.58 \times 10^{15} \log (5.28 \times 10^{16} k^{5} \log^{3} k) \\
& <  1.58 \times 10^{15} (38.51 + 5 \log k + 3 \log (\log k)) \\
& < 2 \times 10^{16} \log k.
\end{align*}
Solving the above inequality gives
\[
k < 1.51 \times 10^{18}.
\]
Substituting this bound of $k$ into \eqref{eq 3.21}, we get $n < 3.1 \times 10^{112}$, which implies that $m < 7.95 \times 10^{112}$.

\noindent
Now, let
\begin{equation}\label{eq 3.28}
\Gamma_{2} := (m+2) \log \alpha - 2n\log \phi + \log \left(c_{\alpha}(\phi + 2 \right)),   
\end{equation}
and $\Lambda_{2} := e^{\Gamma_{2}} - 1$. Then 
\begin{equation}\label{eq 3.29}
    |e^{\Gamma_{2}} -1 | < \frac{7.62}{\phi^{k/2}} < 0.1,
\end{equation}
since $k \geq 360$. Choosing $a := 0.1$, we obtain the inequality
\[
|\Gamma_{2}| = |\log (\Lambda_{2} + 1)| < \frac{- \log(1 - 0.1)}{0.1} \cdot \frac{7.62}{\phi^{k/2}} < \frac{8.1}{\phi^{k/2}}
\]
by Lemma \ref{lem 2.6}. Thus, it follows that
\begin{equation*}
0 < \left| (m+2) \log \alpha - 2n\log \phi + \log \left( c_{\alpha}(\phi + 2) \right) \right| < \frac{8.1}{\phi^{k/2}}.
\end{equation*}
Dividing the above inequality by $\log \phi$, we get 
 \begin{equation}\label{eq 3.30}
 \left| m \frac{\log \alpha}{\log \phi} - 2n + \frac{\log \left(\alpha^{2} c_{\alpha}(\phi + 2) \right)}{\log \phi} \right| < \frac{16.84}{\phi^{k/2}}. 
\end{equation}
To apply Lemma \ref{lem 2.5}, we put
\[ \widehat{\tau} := \frac{\log \alpha}{\log \phi},  \quad \widehat{\mu} :=  \frac{\log \left(\alpha^{2} c_{\alpha}(\phi + 2) \right)}{\log \phi}, \quad A:= 16.84 \quad \text{and} \quad C:= \phi.
\]
If we take $M := 7.95 \times 10^{112}$, which is an upper bound on $m$, we found that $q_{236}$, the denominator of the $236th$ convergent of $\widehat{\tau}$ exceeds $6M$. Furthermore, a quick computation with \textit{Mathematica} gives us the value
\[
\frac{\log(Aq_{236}/\epsilon)}{\log C}
\]
is less than 556. So, if the inequality \eqref{eq 3.30} has a solution, then
\[
\frac{k}{2} < \frac{\log(Aq_{236}/\epsilon)}{\log C} < 556,
\]
which implies that $k \leq 1112$.

\noindent
With the above upper bound for $k$ and by Lemma \ref{lem3.1}, we have
\[
n < 3.1 \times 10^{34} \quad \text{and} \quad m < 7.96 \times 10^{34}.
\]
We apply again Lemma \ref{lem 2.5} to \eqref{eq 3.30} with $M:= 7.96 \times 10^{34}$ and $q = q_{76}$. As a result, we get $k \leq 368$. Hence, we deduce
\[
n < 7.35 \times 10^{31} \quad \text{and} \quad m < 1.89 \times 10^{32}.
\]
In the third application with $M:= 1.89 \times 10^{32}$, we get that $q = q_{72}$ satisfies the conditions of Lemma \ref{lem 2.5} and $k < 342$, which contradicts our assumption that $k > 360$. This completes the proof.

\vspace{05mm} \noindent \footnotesize
\begin{minipage}[b]{90cm}
\large{Department of Mathematics,\\
School of Applied Sciences, \\ 
KIIT University, Bhubaneswar, \\ 
Bhubaneswar 751024, Odisha, India. \\
Email: bptbibhu@gmail.com}
\end{minipage}

\vspace{05mm} \noindent \footnotesize
\begin{minipage}[b]{90cm}
\large{Department of Mathematics,\\
School of Applied Sciences, \\ 
KIIT University, Bhubaneswar, \\ 
Bhubaneswar 751024, Odisha, India. \\
Email: bijan.patelfma@kiit.ac.in; iiit.bijan@gmail.com}
\end{minipage}

\end{document}